\documentclass[12pt]{amsart}

\usepackage[margin=1in]{geometry}

\usepackage{amsmath,amssymb,amsthm,enumerate,mathtools,color,hyperref,cleveref,enumitem}

\title{Improving the trivial bound for $\ell$-torsion in class groups}
\author{Robert J. Lemke Oliver}
\address{Department of Mathematics, Tufts University, Medford, MA 02155}
\email{robert.lemke$\_{ }$oliver@tufts.edu}

\author{Asif Zaman}
\address{Department of Mathematics, University of Toronto, Toronto, Ontario, Canada M5S 2E4}
\email{asif.zaman@utoronto.ca}

\newtheorem{theorem}{Theorem}
\newtheorem{lemma}[theorem]{Lemma}
\newtheorem{corollary}{Corollary}

\theoremstyle{definition} 
\theoremstyle{example} 
\theoremstyle{remark} \newtheorem{remark}[theorem]{Remark}
\newtheorem*{remark*}{Remark}

\crefname{lemma}{Lemma}{Lemmas}
\crefname{theorem}{Theorem}{Theorems}
\crefname{corollary}{Corollary}{Corollaries}
\crefname{remark}{Remark}{Remarks}
\crefname{section}{Section}{Sections}

\numberwithin{equation}{section}
\newcommand{\robert}[1]{{\color{red} \sf $\clubsuit\clubsuit\clubsuit$ Robert: [#1]}}
\newcommand{\asif}[1]{{\color{blue} \sf $\clubsuit\clubsuit\clubsuit$ Asif: [#1]}}

\newcommand{\N}{\mathrm{N}}

\renewcommand{\epsilon}{\varepsilon}
\renewcommand{\leq}{\leqslant}
\renewcommand{\geq}{\geqslant}

\newcommand{\C}{\mathbb{C}}
\newcommand{\R}{\mathbb{R}}
\newcommand{\Q}{\mathbb{Q}}

\newcommand{\km}{\mathfrak{m}}
\newcommand{\kn}{\mathfrak{n}}
\newcommand{\kp}{\mathfrak{p}}

\newcommand{\hidecomments}{
	\renewcommand{\asif}[1]{}
	\renewcommand{\robert}[1]{}
}

\hidecomments

\begin{document}

	\begin{abstract}

		For any number field $K$ with $D_K=|\mathrm{Disc}(K)|$ and any integer $\ell \geq 2$, we improve over the commonly cited trivial bound $|\mathrm{Cl}_K[\ell]| \leq |\mathrm{Cl}_K| \ll_{[K:\Q],\epsilon} D_K^{1/2+\epsilon}$ on the $\ell$-torsion subgroup of the class group of $K$ by showing that $|\mathrm{Cl}_K[\ell]| = o_{[K:\Q],\ell}(D_K^{1/2})$.  In fact, we obtain an explicit log-power saving.  This is the first general unconditional saving over the trivial bound that holds for all $K$ and all $\ell$.	
	
	\end{abstract}
	
	\maketitle

%%%%%%%%%
%%%%%%%%%
%%%%%%%%%	
\section{Introduction}
%%%%%%%%%
%%%%%%%%%
%%%%%%%%%

Let $K \neq \Q$ be a number field of degree $[K:\mathbb{Q}]$ and absolute discriminant $D_K = |\mathrm{Disc}(K)|$. Let $\mathrm{Cl}_K$ be the class group of $K$. For integers $\ell \geq 2$, the $\ell$-torsion of the class group satisfies 
\begin{equation} \label{eqn:OldTrivialBound}
|\mathrm{Cl}_K[\ell]| \leq |\mathrm{Cl}_K| \ll_{[K:\mathbb{Q}]} D_K^{1/2} (\log D_K)^{[K:\mathbb{Q}]-1}
\end{equation}
by a result of Landau.  This trivial bound can be refined. By the class number formula, we have that 
\begin{equation} \label{eqn:ClassNumberFormula}
	|\mathrm{Cl}_K| \asymp_{[K:\Q]} \frac{\kappa_K D_K^{1/2}}{R_K}, 
\end{equation}
where $\kappa_K$ is the residue of the Dedekind zeta function $\zeta_K(s)$ at $s=1$, and $R_K$ is the regulator of $K$. A theorem of Landau implies that 
\begin{equation} \label{eqn:ResidueBound}
\kappa_K \ll_{[K:\Q]} (\log D_K)^{[K:\Q]-1}, 
\end{equation}
and a theorem of Silverman \cite{Silverman} (cf. Akhtari--Vaaler \cite{Akhtari-Vaaler-Regulators}) implies that 
\begin{equation} \label{eqn:RegulatorBound}
R_K \gg_{[K:\Q]} (\log D_K)^{r_K-\rho_K}, 
\end{equation}
where $r_K$ is the rank of the unit group of $K$ and $\rho_K$ is the maximum rank of the unit group over all proper subfields of $K$. This combination yields a refinement over \eqref{eqn:OldTrivialBound} of the form 
\begin{equation}  \label{eqn:TrivialBound}
|\mathrm{Cl}_K[\ell]| \leq  |\mathrm{Cl}_K| \ll_{[K:\mathbb{Q}]} D_K^{1/2} (\log D_K)^{[K:\Q]-r_K+\rho_K-1}.  
\end{equation}
Improvements over this refined trivial bound\footnote{The literature on $\ell$-torsion in class groups  typically refers to the ``trivial bound'' as either Landau's bound \eqref{eqn:OldTrivialBound} or the cruder bound $|\mathrm{Cl}_K[\ell]| \ll_{[K:\Q],\epsilon} D_K^{1/2+\epsilon}$ for $\epsilon > 0$.}
have generated significant interest in many cases, and often take a power-savings form
\begin{equation} \label{eqn:PowerSavings}
|\mathrm{Cl}_K[\ell]| \ll_{[K:\Q],\ell,\Delta} D_K^{\Delta} 
\end{equation}
for some fixed $0 < \Delta < 1/2$ whose value depends on $\ell$, $[K:\Q]$, and the subfield structure of $K$. The so-called  ``$\ell$-torsion conjecture'' posits that \textit{any} fixed $\Delta > 0$ is admissible. This question was first asked by Brumer--Silverman \cite{Brumer-Silverman-NumberECs}.

Unconditional pointwise progress towards this conjecture has been restricted to certain degrees $[K:\Q]$,  integers $\ell$, or subfield structures.  Table \ref{table} below summarizes the progress made by Bhargava--Shankar--Taniguchi--Thorne--Tsimerman--Zhao \cite{BSTTZ-TwoTorsion}, Ellenberg--Venkatesh \cite{Ellenberg-Venkatesh-ReflectionPrinciples}, Gauss,  Heath-Brown \cite{HeathBrown-ClassGroups}, Helfgott--Venaktesh \cite{Helfgott-Venaktesh-ThreeTorsion}, Kl\"{u}ners--Wang \cite{Kluners-Wang-ClassGroup}, Pierce \cite{Pierce-3Part,Pierce-3Part-SquareSieve}  and Wang \cite{Wang-PointwiseAbelian,Wang-PointwiseNilpotent}.  There has also been substantial interest in unconditional progress \textit{on average} over families of number fields, such as \cite{HeathBrown-Pierce-Averages,Ellenberg-Pierce-Wood-ClassGroups,Widmer-Bounds,Frei-Widmer-CyclicTorsion,Pierce-TurnageButterbaugh-Wood-EffectiveChebotarev,An-DihedralTorsion,Frei-Widmer-Moments,Thorner-Zaman-DedekindZeta,LemkeOliver-Wang-Wood-ThreeTorsion,LemkeOliver-Thorner-Zaman-ApproximateArtin,Koymans-Thorner-Moments,LemkeOliver-Smith-FaithfulArtin}.    For further discussion, we refer the reader to the excellent survey article of Pierce \cite{Pierce-ICM2022} and very recent work of Heath-Brown \cite{HeathBrown-ClassGroups}.

\def\arraystretch{1.2}
{\small\begin{table}[h!]
\caption{Pointwise progress towards $\ell$-torsion conjecture according to \eqref{eqn:PowerSavings} }
\label{table} 
\begin{tabular}{|l|l|c|c|l|l|}
\hline  \sc Source & $\Delta >$ & $[K:\Q]$ & $\ell$ & \sc Other restrictions  \\
\hline Gauss &  $0$ & $2$ & $2$ &  \\
\hline \cite{Pierce-3Part,Pierce-3Part-SquareSieve} & $1/2-1/56$ & $2$ & $3$ & \\
\hline \cite{Helfgott-Venaktesh-ThreeTorsion} & $1/2-0.0582...$ & $2$ & $3$ & \\
\hline \cite{Ellenberg-Venkatesh-ReflectionPrinciples} & $1/2-1/6$ & $2,  3 $ & $3$ & \\
\hline  \cite{HeathBrown-ClassGroups} & $1/2-1/2\ell$ & $2,3$ &  prime  &  $p \mid D_K \implies p \leq D_K^{\delta_\ell}$  \\
\hline  \cite{HeathBrown-ClassGroups} & $1/2-1/4\ell$ & $3$ &   prime   & pure cubic extension \\
\hline  \cite{BSTTZ-TwoTorsion} & $1/2-0.2215...$ & $3,4$ & $2$ &  \\
\hline  \cite{Ellenberg-Venkatesh-ReflectionPrinciples}  &  $1/2-1/168$  &  $4$  &  $3$  &  $\mathrm{Gal}(\widetilde{K}/\mathbb{Q}) \cong A_4$ or $S_4$  \\
\hline  \cite{Ellenberg-Venkatesh-ReflectionPrinciples}  &  $1/2-\delta$  &  $4$  &  $3$  &  $\mathrm{Gal}(\widetilde{K}/\mathbb{Q}) \cong C_4, V_4,$ or $D_4$ \\
\hline  \cite{BSTTZ-TwoTorsion} & $1/2-1/{2[K:\Q]}$ & $\geq 5$ & $2$ &  \\
\hline  \cite{Kluners-Wang-ClassGroup} & $0$ & $p^r$ &   $p^s$    & $\mathrm{Gal}(\widetilde{K}/\mathbb{Q}) \cong G$, any $p$-group $G$  \\
\hline  \cite{Wang-PointwiseAbelian,Wang-PointwiseNilpotent} & $1/2-\delta_{G,\ell}$ &  powerful  & $\geq 2$  & $\mathrm{Gal}(K/\mathbb{Q}) \cong G$, many\footnotemark\,nilpotent groups $G$  \\
\hline 
\end{tabular}
\end{table}}%

\footnotetext{More precisely, this includes all nilpotent groups $G$ where each Sylow subgroup is not cyclic or quaternion.}

The purpose of this article is to give a uniform improvement over  \eqref{eqn:TrivialBound} weaker than \eqref{eqn:PowerSavings}. 
\begin{corollary} \label{cor:Main}
If $K \neq \mathbb{Q}$ is a number field, and $\ell \geq 2$ is an integer, then 
\[
|\mathrm{Cl}_K[\ell]| \ll_{[K:\mathbb{Q}],\ell} D_K^{1/2} (\log D_K)^{-r_K+\rho_K-1} (\log\log D_K)^{3[K:\Q]/2}.
\]
\end{corollary}

This result follows immediately from a new inverse relationship between the size of the $\ell$-torsion subgroup of the class group and the class group itself, provided that the class group is close to extremal size.

\begin{theorem} \label{thm:Main} 
	Let $K \neq \mathbb{Q}$ be a number field. Let $\ell \geq 2$ be an integer. There exists a sufficiently small constant $\delta = \delta_{[K:\Q],\ell}  > 0$  such that if  
	\begin{equation}\label{eqn:ClassNumberAssumption}
	|\mathrm{Cl}_K| = V_{K}^{[K:\Q]} D_K^{1/2} (\log D_K)^{-r_K+\rho_K-1} (\log\log D_K)^{3[K:\Q]/2}, 
	\end{equation}
	for some $V_K \geq [K:\Q]/\delta$ then  
	\[
	|\mathrm{Cl}_K[\ell]| \ll_{[K:\mathbb{Q}],\ell}  |\mathrm{Cl}_K| (\log D_K)^{-\delta V_K} .
	\]
\end{theorem}
\begin{remark*}
Notice $V_K \ll_{[K:\Q]}  (\log D_K) (\log\log D_K)^{-3/2}$ according to \eqref{eqn:TrivialBound}. 
\end{remark*}

Our argument is soft and based on a particularly fruitful lemma of Ellenberg--Venkatesh. The basic principle is that if there are \textit{many} small non-inert unramified primes in $K$, then we save over the trivial bound. A refined version of their lemma (\cref{lem:EV}) illustrates a contrasting   principle: if the residue of the Dedekind zeta function $\zeta_K(s)$ is small, then  we also save over the trivial bound. We exploit that this residue  is small precisely when $K$ has \textit{few} small non-inert unramified  primes. By balancing the effects of these conflicting principles, we unconditionally save over the trivial bound. 
%~ 
The residue's relationship to the class number via \eqref{eqn:ClassNumberFormula} is our source for the inverse relationship in \cref{thm:Main}.  
%~ 

This observation also has conditional consequences which illustrate its limitations. The classic conditional benchmark is again due to  Ellenberg and Venkatesh \cite{Ellenberg-Venkatesh-ReflectionPrinciples} who showed that the generalized Riemann hypothesis (GRH) implies any fixed 
\begin{equation} \label{eqn:GRH-EV}
\Delta > \frac{1}{2}-\frac{1}{2\ell ([K:\Q]-1)}	
\end{equation}
is admissible in \eqref{eqn:PowerSavings}.  Heath-Brown \cite{HeathBrown-ClassGroups} very recently showed how one might relax their GRH hypothesis with a subconvexity hypothesis to deduce a number of intriguing results.  While his focus was on quadratic and cubic fields $K$, the argument applies to any number field as remarked therein. Combined with our observations, we obtain the following result.
 
\begin{theorem} \label{thm:Subconvexity}
 Let $K \neq \Q$ be a number field. Fix an integer $\ell \geq 2$ and real $\Delta > \frac{1}{2}-  \frac{1}{2\ell ([K:\Q]-1)}$. If there exists  $A \geq 1$ and $0 < \eta < 1$  such that  
 \begin{equation}
 \label{eqn:Assumption}
 |\zeta_K(\tfrac{1}{2}+it)| \ll_{[K:\Q],\ell,A,\eta} D_K^{\frac{1-\eta}{4\ell ([K:\Q]-1)}} |1+it|^A \quad \text{ for } t \in \R, 
 \end{equation}
then  
\[
|\mathrm{Cl}_K[\ell]| \ll_{[K:\mathbb{Q}],\ell,A,\eta,\Delta} D_K^{\Delta}. 
\]
The implied constant is  effective unless $K$ has a quadratic subfield and $\Delta \leq  \frac{1}{2}-\frac{\eta}{4\ell ([K:\Q]-1)}$.
\end{theorem}

\cref{thm:Subconvexity} recovers \eqref{eqn:GRH-EV} under a subconvexity hypothesis which is notably weaker than GRH or the generalized Lindel\"{o}f hypothesis (and one can weaken our hypothesis even further as shown in \cite{HeathBrown-ClassGroups}), but the implied constants are ineffective if $K$ contains a quadratic subfield. Indeed, as shown in \cref{sec:SubconvexityProof}, this ineffective form follows quickly from Heath-Brown's arguments and the ineffective Brauer--Siegel theorem. By supplementing this argument with our observations leading to \cref{thm:Main}, we are able to deduce the effective power savings described in \cref{thm:Subconvexity} for all number fields. See \cref{rem:Effective} for more details on achieving effective power savings. 
 
%%%%%%%%%
%%%%%%%%%
%%%%%%%%%	
\section*{Acknowledgements}
	 We would like to thank Roger Heath-Brown, Greg Martin, Lillian Pierce, Paul Pollack, Arul Shankar,  Jesse Thorner, Jacob Tsimerman, Caroline Turnage-Butterbaugh, and Jiuya Wang  for their encouragement and comments.  RJLO was partially supported by NSF grant DMS-2200760.  AZ was partially supported by NSERC grant RGPIN-2022-04982. 
%%%%%%%%%
%%%%%%%%%
%%%%%%%%%

%%%%%%%%%
%%%%%%%%%
%%%%%%%%%
\section{Preliminaries}
	\label{sec:prelim}
%%%%%%%%%
%%%%%%%%%
%%%%%%%%%
For a number field $K$, denote:
\begin{itemize}
	\item $[K:\Q]$ to be the degree of $K$
	\item $\mathcal{O}_K$ to be the ring of integers of $K$
	\item $\N = \mathrm{Nm}_{K/\Q}$ to be the absolute norm on ideals
	\item $\mathfrak{d}_K$ to be the absolute different ideal  of $K$ 
	\item $D_K = \N\mathfrak{d}_K = |\mathrm{Disc}(K/\Q)|$ to be the norm of the absolute different ideal
	\item $\zeta_K(s)$ to be the Dedekind zeta function of $K$
	\item $\kappa_K$ to be the residue of $\zeta_K(s)$ at $s=1$
\end{itemize}
The letters $\km, \kn$ will be reserved to denote integral ideals of $K$  whereas $\kp$   will be reserved for prime ideals of $K$. Similarly, letters $m,n$ will denote rational integers and the letter $p$ will denote a rational prime.  

We begin by recording the key lemma of Ellenberg--Venkatesh including a refinement involving $\kappa_K$ and a new observation of Heath-Brown \cite[Lemma 3]{HeathBrown-ClassGroups}.

\begin{lemma} \label{lem:EV}
Let $K \neq \Q$ be a number field. Fix an integer $\ell \geq 2$ and $0 < \eta < 1$. If $M$ is the number of integral ideals  $\kn$ of $K$ such that $\kn$ is relatively prime to the different ideal $\mathfrak{d}_K$, its norm $\N\kn$ is a squarefree integer, and $\mathrm{N}\kn \leq   D_K^{\frac{1-\eta}{2\ell([K:\Q]-1)}}$,  then 
\[
|\mathrm{Cl}_K[\ell]| \ll_{\eta,\ell,[K:\mathbb{Q}]} \frac{\kappa_K \sqrt{D_K}}{M}  . 
\]
\end{lemma}
\begin{proof}
Using the notation of Lemma 3 in Ellenberg--Venkatesh \cite{Ellenberg-Venkatesh-ReflectionPrinciples}, they showed
\[
|\mathrm{Cl}_K[\ell]| \ll_{\eta,\ell,[K:\mathbb{Q}]} \frac{\mathrm{vol}(\widetilde{\mathrm{Cl}_K})}{\mathrm{vol}(P_{\ell}/P)}.
\]
They estimate  $\mathrm{vol}(\widetilde{\mathrm{Cl}_K})$, the volume of the Arakelov class group, from above by the class number of $K$ times the regulator of $K$, times some constant depending at most on the degree of $K$. By the class number formula, this quantity is at most $\ll_{[K:\mathbb{Q}]} \kappa_K \sqrt{D_K}$.  The volume $\mathrm{vol}(P_{\ell}/P)$ was shown to be bounded below by the number of  prime ideals $\mathfrak{p}$ of degree one which do not ramify over $\mathbb{Q}$ with norm $\mathrm{N}\mathfrak{p} \leq D_K^{\frac{1-\eta}{2\ell ([K:\Q]-1)}}$. Heath-Brown observed that $\mathrm{vol}(P_{\ell}/P)$ can instead be bounded below by the number of integral ideals $\mathfrak{n}$ satisfying the assumptions stated in the lemma.  
\end{proof}

The Dedekind zeta function $\zeta_K(s)$ can be written in two forms. First, it can be written over the number field $K$ using ideals as 
\begin{equation} \label{eqn:DedekindZeta-K}
\zeta_K(s) = \sum_{\kn \subseteq \mathcal{O}_K} (\N\kn)^{-s} = \prod_{\kp} \Big(1 - (\N\kp)^{-s} \Big)^{-1} \quad \text{ for } \Re(s) > 1,
\end{equation}
where the product runs over prime ideals $\kp$ of $K$.  Second, it can be written over the base field $\Q$ using integers  as
\begin{equation} \label{eqn:DedekindZeta-Q}
\zeta_K(s) = \sum_{n=1}^{\infty} \lambda_K(n) n^{-s} = \prod_{p} \Big(1 + \sum_{j=1}^{\infty} \frac{\lambda_K(p^j)}{p^{js}} \Big) \quad \text{ for } \Re(s) > 1,
\end{equation}
where $\lambda_K$ is a multiplicative function on the positive integers satisfying  
\begin{equation} \label{eqn:Coefficients}
\lambda_K(n) = \sum_{\substack{\N\kn = n } } 1 \quad \text{ and } \quad \lambda_K(p^j) \leq [K:\Q]^j  	
\end{equation}
for all $n,j \geq 1$ and primes $p$. The second equality holds because at most $[K:\Q]$ prime ideals $\kp$ lie above a rational prime $p$. We shall use the latter form \eqref{eqn:DedekindZeta-Q} to count ideals with squarefree norm and without prime ideals dividing the different. 

Let $\mu$ be the classical M\"{o}bius function on the integers, so $\mu^2$ is the indicator function on squarefree integers. Define the non-negative multiplicative function $\lambda^{\flat}_K$ by
\begin{equation} \label{eqn:Coefficients-Flat}
\lambda^{\flat}_K(n) := \mu^2(n) \sum_{\substack{\N\kn = n \\ (\kn,\mathfrak{d}_K)=1} } 1, \quad \text{ so} \quad 0 \leq \lambda^{\flat}_K(n) \leq \lambda_K(n)
\end{equation}
 for all $n \geq 1$. At all primes $p$, we have that 
\begin{equation} \label{eqn:PrimeBound}
0 \leq \lambda_K(p)-\lambda^{\flat}_K(p) \leq  \frac{[K:\Q]}{2}
\end{equation}
since 
\[
\lambda_K(p)-\lambda^{\flat}_K(p) = |\{ \kp   :  \N\kp = p \text{ and } \kp \mid \mathfrak{d}_K  \}| = \sum_{\substack{ \N\kp = p \\ e_{\kp} \geq 2} } 1 \leq \frac{1}{2} \sum_{\substack{ \kp \subseteq (p) } } e_{\kp} f_{\kp}  = \frac{[K:\Q]}{2}. 
\]
where $e_{\kp}$ and $f_{\kp}$ are respectively the ramification index and inertia degree of $\kp$ over $(p)$. In particular, at unramified primes, 
\begin{equation} \label{eqn:UnramifiedBound}
	\lambda^{\flat}_K(p) = \lambda_K(p) \quad \text{when $p \nmid D_K$.}
\end{equation}

\begin{lemma} \label{lem:Sift}
	Let $K \neq \Q$ be a number field. For $x \geq 3$ and $s \in \C$, define 
	\begin{equation} \label{eqn:H}
	H_K(s, x) := \prod_{p \leq x} \Big(1+\frac{\lambda^{\flat}_K(p)}{p^s}\Big) \Big( 1+ \sum_{j=1}^{\infty} \frac{\lambda_K(p^j)}{p^{js}} \Big)^{-1}.
	\end{equation}
	Then  $H_K(s,x)$ is entire, 
	\[
	H_K(1,x) \gg_{[K:\Q]} 	(\log \log D_K)^{-[K:\Q]/2},
	\]
	and there exists a constant $C = C_{[K:\Q]} > 0$ such that for $t \in \R$,  
	\[
	|H_K(\tfrac{1}{2}+it,x)| \ll_{[K:\Q]} \big( e^{(\log D_K)^{1/2} } \log x \big)^{C}.
	\]
\end{lemma}
\begin{proof}  Since 
\[
\Big( 1+ \sum_{j=1}^{\infty} \frac{\lambda_K(p^j)}{p^{js}} \Big)^{-1} = \prod_{\kp \mid (p)} \big(1 - (\N\kp)^{-s} \big),
\]
the function $H_K(s,x)$ is a finite product of entire functions and hence entire. For the bounds on $H_K(s,x)$, we start with an estimate for each local factor. For all primes $p$ and $\Re(s) > 1/4$,  we have uniformly that  
\[
\Big(1+\frac{\lambda^{\flat}_K(p)}{p^s}\Big) \Big( 1+ \sum_{j=1}^{\infty} \frac{\lambda_K(p^j)}{p^{js}} \Big)^{-1} = 1 - \frac{\lambda_K(p)-\lambda^{\flat}_K(p)}{p^s} + O_{[K:\Q]}\big(p^{-2\Re(s)}\big)
\]
by \eqref{eqn:DedekindZeta-Q} and the inequality $0 \leq \lambda_K^{\flat}(p) \leq \lambda_K(p)$ from \eqref{eqn:Coefficients-Flat}. The expansion $\log(1+u) = u + O(u^2)$ for $|u| < 1/2$ gives 
	\begin{equation} \label{eqn:H-Estimate}
	H_K(s,x) 
		 = \exp\Big( -\sum_{\substack{p \leq x} } \frac{\lambda_K(p)-\lambda^{\flat}_K(p)}{p^s} + O_{[K:\Q]}\Big(\sum_{p \leq x} p^{-2\Re(s)}\Big) \Big). 
	\end{equation}
At $s=1$, we have by \eqref{eqn:PrimeBound},\eqref{eqn:UnramifiedBound}, and the prime number theorem that 
	\[
	\sum_{p \leq x} \frac{\lambda_K(p)-\lambda^{\flat}_K(p)}{p} \leq \frac{[K:\Q]}{2} \sum_{p \mid D_K} \frac{1}{p} \leq \frac{[K:\Q]}{2} \log\log \log D_K + O_{[K:\Q]}(1). 
	\]
 The lower bound for $H_K(1,x)$ now follows from \eqref{eqn:H-Estimate} with the observation $\sum_{p \leq x} p^{-2} \ll 1$. At $s = \frac{1}{2}+it$ for  $t \in \R$, we again have by \eqref{eqn:PrimeBound}, \eqref{eqn:UnramifiedBound}, and the prime number theorem that 
	\[
	\sum_{p \leq x} \frac{|\lambda_K(p)-\lambda^{\flat}_K(p)|}{p^{1/2}} \leq \frac{[K:\Q]}{2} \sum_{p \mid D_K} \frac{1}{p^{1/2}}   \ll_{[K:\Q]}  (\log D_K)^{1/2} . 
	\]
	The last estimate is crude but sufficient. The upper bound for $|H_K(\tfrac{1}{2}+it,x)|$ now follows from \eqref{eqn:H-Estimate} and the estimate $\sum_{p \leq x} p^{-1} \ll \log\log x$ following from Mertens' formula. 
\end{proof}

We also record a classic convexity estimate and an explicit test function.

\begin{lemma} \label{lem:Convexity}
	Let $K$ be a number field. Fix $\delta > 0$. For $t \in \R$,
	\[
	\big| \zeta_K(\tfrac{1}{2}+it)\big| \ll_{[K:\Q],\delta} \big( D_K |1+it|^{[K:\mathbb{Q}]} \big)^{1/4+\delta}. 
	\] 
\end{lemma}
\begin{proof}
	See, for example, Rademacher \cite{Rademacher-PhragmenLindelof}.
\end{proof}

\begin{lemma} \label{lem:weight} 
For an integer $k \geq 1$, define $\varphi_{k} : [0,\infty) \to [0,\infty)$ by
\[
\varphi_{k}(t) = \begin{cases} \dfrac{t \log(1/t)^{k}}{k!}
 	& 0 < t \leq  1, \\
 	0 & t =0 \text{ or } t > 1.
\end{cases}
\]
Its maximum value is $\varphi_k(e^{-k}) \leq 1$ and its Mellin transform $\widehat{\varphi_{k}}(s)$  is given by 
\[
\widehat{\varphi_{m}}(s) = \frac{1}{(s+1)^{k+1}} \quad \text{ for } s \in \mathbb{C} \setminus \{-1\}. 
\]
\end{lemma}
\begin{proof}  This can be verified by a routine induction and calculus.  
\end{proof}

%%%%%%%%%
%%%%%%%%%
%%%%%%%%%
\section{Proof of Theorem 1}
\label{sec:MainProof}
%%%%%%%%%
%%%%%%%%%
%%%%%%%%%

Let $K \neq \Q$ be a number field and $\ell \geq 2$ be an integer.   Fix 
\begin{equation}
\eta = 1/2, \quad  y = D_K^{\frac{1-\eta}{2\ell([K:\Q]-1)}} 
\quad 
\text{and} 
\quad 
x = y^{8 \ell [K:\Q]}. 
\label{eqn:FixParameters}
\end{equation}
Note  $D_K^2 \leq x \leq D_K^3$. 
Using \eqref{eqn:Coefficients-Flat}, define the prime ideal counting function
\begin{align*}
\pi^{\flat}_K(y) & := \sum_{p \leq y} \lambda^{\flat}_K(p) = |\{ \kp  : \N\kp \leq y,  \kp \nmid \mathfrak{d}_K , \N\kp = p \text{ for some prime } p \}| 
\end{align*}
so, by a less refined application of \cref{lem:EV}, we have that
\[
|\mathrm{Cl}_K[\ell]| \ll_{[K:\Q],\ell} \frac{\kappa_K D_K^{1/2}}{1+\pi^{\flat}_K(y)} . 
\]
Let $\pi_{\Q}$ be the classical prime counting function. Observe $0 \leq \pi^{\flat}_K(y) \leq [K:\mathbb{Q}] \pi_{\Q}(y)$ by \eqref{eqn:Coefficients} and \eqref{eqn:Coefficients-Flat}, so 
\begin{equation}
[K:\mathbb{Q}] (\pi_{\Q}(z) - 1) \leq \pi^{\flat}_K(y) \leq  [K:\mathbb{Q}] \pi_{\Q}(z) \qquad \text{for some} \quad 2 \leq z \leq y. 
\label{eqn:FewPrimes}
\end{equation}
Since $\pi_{\Q}(z) \gg z/\log z$, it follows that 
\begin{equation} \label{eqn:Torsion-ResidueBound}
|\mathrm{Cl}_K[\ell]| \ll_{[K:\Q],\ell}    \frac{\log z}{z} \cdot \kappa_K D_K^{1/2}.
\end{equation}
Our goal is therefore to estimate $\kappa_K$ in terms of the unknown parameter $z$, and then take the supremum of the above bound over all possible values of $2 \leq z \leq y$. 

We begin by defining the sum 
\[
S(x) := \sum_{n \leq x}  \lambda^{\flat}_K(n)   \varphi_k(n/x),
\]
where $k = [K:\Q]-1 \geq 1$ and $\varphi_k$ is given by \cref{lem:weight}.
By \eqref{eqn:DedekindZeta-Q}, \eqref{eqn:Coefficients-Flat}, and \eqref{eqn:H}, it follows by Mellin inversion that 
\begin{align*}
S(x) & = \frac{1}{2\pi i} \int_{(2)} \zeta_K(s) H_K(s,x) \frac{x^s}{(s+1)^{[K:\Q]}} ds \\
& = \frac{\kappa_K H_K(1,x)}{2^{[K:\mathbb{Q}]}} x  + \frac{1}{2\pi} \int_{-\infty}^{\infty}  \zeta_K(\tfrac{1}{2}+it) H_K(\tfrac{1}{2}+it,x)\frac{x^{1/2+it}}{(3/2+it)^{[K:\Q]}} dt . 
\end{align*}
Applying \cref{lem:Convexity} with $\delta = \frac{1}{8}$, the upper bound on $H_K(\frac{1}{2}+it,x)$ from \cref{lem:Sift},  and $\int_{-\infty}^{\infty}  (9/4+ t^2 )^{-\frac{5}{16}[K:\mathbb{Q}]} dt \ll 1$, we deduce for some constant $C = C_{[K:\Q]} > 0$ that
\[
S(x) =  \frac{\kappa_K H_K(1,x)}{2^{[K:\mathbb{Q}]}} x  + O_{[K:\Q]}\Big(  D_K^{3/8} x^{1/2} \big(  e^{(\log D_K)^{1/2}} \log x \big)^{C}  \Big). 
\]
Rearranging and applying the lower bound on $H_K(1,x)$ from \cref{lem:Sift}, we find that
\begin{equation} \label{eqn:Residue}
\begin{aligned}
\kappa_K 
& \ll_{[K:\Q]} \Big( \frac{S(x)}{x} +   \frac{D_K^{3/8}}{x^{1/2}} \big( e^{(\log D_K)^{1/2}} \log x \big)^{C}  \Big) (\log\log D_K)^{[K:\Q]/2} \\
& \ll_{[K:\Q]}  \frac{S(x)}{x} (\log\log D_K)^{[K:\mathbb{Q}]/2}  +  D_K^{-1/2}
\end{aligned}
\end{equation}
where the secondary terms have been absorbed into the $D_K^{-1/2}$ term since $D_K^2 \leq x \leq D_K^3$. It remains to estimate $S(x)$.  

As $\varphi_{k}(t) \leq 1$ for all $t \geq 0$ by \cref{lem:weight} and $\lambda^{\flat}_K$ is multiplicative, we may write 
\begin{equation}
S(x) 
\leq \sum_{n \leq x} \lambda^{\flat}_K(n) = \sum_{\substack{m \leq x \\ p \mid m   \implies p \leq y} }  \, \lambda^{\flat}_K(m) \, \sum_{\substack{ n \leq x/m \\ p \mid n \implies p > y } }  \lambda^{\flat}_K(n).
\label{eqn:KeyStep1}
\end{equation}
For every squarefree integer $n \leq x$ free of prime factors less than $y$, we have by multiplicativity of $\lambda^{\flat}_K$ that 
\[
\lambda^{\flat}_K(n) = \prod_{p \mid n} \lambda^{\flat}_K(p) \leq [K:\Q]^{8\ell [K:\Q]} 
\]
since $\lambda^{\flat}_K(p) \leq [K:\Q]$ by \eqref{eqn:Coefficients} and \eqref{eqn:Coefficients-Flat}, and $|\{ p : p \mid n\}| \leq 8\ell [K:\Q]$ by \eqref{eqn:FixParameters}. For fixed $m \leq x$, it follows that
\[
 \sum_{\substack{ n \leq x/m \\ p \mid n \implies p > y } } \lambda^{\flat}_K(n)   \ll_{[K:\Q], \ell}   \sum_{\substack{n \leq x/m \\ p \mid n    \implies p > y } } 1  
\ll_{[K:\Q], \ell}  \mathbf{1}_{\{ m \leq x\}} + \frac{x}{\log y} \frac{\mathbf{1}_{\{ m \leq x/y\}}}{m}. 
\]
Overall, as $\log y \asymp_{[K:\Q],\ell} \log x$ by \eqref{eqn:FixParameters}, this implies by \eqref{eqn:KeyStep1}  that 
\begin{equation} \label{eqn:S-split}
S(x) \ll_{[K:\Q], \ell} \sum_{\substack{m \leq x \\ p \mid m \implies p \leq y} } \lambda^{\flat}_K(m) + \frac{x}{\log x} \sum_{\substack{m \leq x/y \\ p \mid m \implies p \leq y} } \frac{\lambda^{\flat}_K(m)}{m}.
\end{equation}

We proceed to estimate the righthand and lefthand sum in \eqref{eqn:S-split}. The righthand sum can be bounded via \eqref{eqn:FewPrimes},  the inequalities $\lambda^{\flat}_K(p) \leq [K:\Q]$ and  $\log(1+u) < u$ for $u > 0$, and the prime number theorem to yield
\[
\sum_{\substack{m \leq x/y  \\ p \mid m \implies p \leq y} } \frac{\lambda^{\flat}_K(m)}{m} \leq \prod_{\substack{ p \leq y } } \Big(1 + \frac{\lambda^{\flat}_K(p)}{p}\Big) 
\leq \exp\Big(\sum_{p \leq z}  \frac{[K:\Q]}{p} \Big)  \ll_{[K:\Q]} (\log z)^{[K:\Q]}.
\]
For the lefthand sum, we apply Rankin's trick with a yet-to-be-specified parameter $\frac{3}{4} \leq \alpha < 1$ and use \eqref{eqn:FewPrimes} again to see that 
\[
\sum_{\substack{m \leq x \\ p \mid m \implies p \leq y} } \lambda^{\flat}_K(m)
\leq \sum_{\substack{m \leq x \\ p \mid m \implies p \leq y} }\lambda^{\flat}_K(m)   \Big( \frac{x}{m} \Big)^\alpha
\leq x^{\alpha} \prod_{\substack{ p \leq y} } \Big( 1 + \frac{\lambda^{\flat}_K(p)}{p^{\alpha}} \Big). 
\]
We similarly have   that  
\[
\prod_{\substack{ p \leq y} } \Big( 1 + \frac{\lambda^{\flat}_K(p)}{p^{\alpha}} \Big) \leq \exp\Big(\sum_{p \leq z}  \frac{[K:\Q]}{p^{\alpha}} \Big)  \leq \exp\Big( O\Big( \frac{[K:\Q]  z^{1-\alpha}}{(1-\alpha) \log z}\Big)  \Big)
\]
by the prime number theorem.  

Collecting these estimates in \eqref{eqn:S-split}, we conclude that 
\begin{align*}
S(x)
&  \ll_{[K:\Q],\ell} x \exp\Big( - (1-\alpha) \log x + O\Big(\frac{[K:\Q] z^{1-\alpha}}{(1-\alpha) \log z} \Big) \Big) + \frac{x}{\log x}  (\log z)^{[K:\Q]} \\
&   \ll_{[K:\Q],\ell} x \exp\Big( -  \frac{\log x}{\log z} \Big) + \frac{x}{\log x}  (\log z)^{[K:\Q]} \\
& \ll_{[K:\Q],\ell}  \frac{x}{\log x}  (\log z)^{[K:\Q]}
\end{align*}
upon choosing $\alpha = \max\{ 1 -  (\log z)^{-1}, 3/4 \}$ in the second step. Note the third step applies the weak bound $e^{-u} \ll u^{-1}$ for $u = \frac{\log x}{\log z} \geq 8\ell [K:\Q]$ as $2 \leq z \leq y = x^{1/8\ell [K:\Q]}$ by \eqref{eqn:FixParameters}. 
 By \eqref{eqn:Residue} and the observation $D_K^2 \leq x \leq D_K^3$ from \eqref{eqn:FixParameters}, it follows that 
\[
\kappa_K \ll_{[K:\Q],\ell}  (\log z)^{[K:\Q]} (\log D_K)^{-1}  (\log\log D_K)^{[K:\Q]/2}
\]
for $2 \leq z \leq y = D_K^{\frac{1}{4 \ell ([K:\Q]-1)}}$. From \eqref{eqn:ClassNumberFormula} and \eqref{eqn:RegulatorBound}, we have that 
\[
|\mathrm{Cl}_K| \ll_{[K:\Q],\ell} (\log z)^{[K:\Q]} \cdot D_K^{1/2} (\log D_K)^{-r_K+\rho_K-1} (\log\log D_K)^{[K:\Q]/2}.
\]
Since $V_K$ satisfies \eqref{eqn:ClassNumberAssumption}, this implies $V_K \log\log D_K \ll_{[K:\Q],\ell} \log z$ in which case
\[
z \geq (\log D_K)^{3 \delta V_K},
\]
where $\delta = \delta_{[K:\Q],\ell} > 0$ is sufficiently small. 

Combining these observations with \eqref{eqn:Torsion-ResidueBound}, we obtain that 
 \begin{align*}
|\mathrm{Cl}_K[\ell]| 
&  \ll_{[K:\Q],\ell} z^{-1} (\log z)^{[K:\Q]+1} D_K^{1/2} (\log D_K)^{-1}  (\log\log D_K)^{[K:\Q]/2}  
\end{align*}
for some unknown $z$ satisfying $(\log D_K)^{3\delta V_K} \leq z \leq  y$. The supremum  of this bound occurs when $z \asymp_{[K:\Q],\ell} (\log D_K)^{3\delta V_K}$. We conclude  that 
\begin{align*}
|\mathrm{Cl}_K[\ell]| 
& \ll_{[K:\Q],\ell}  V_K^{[K:\Q]+1}  D_K^{1/2} (\log D_K)^{-2\delta V_K - r_K+\rho_K-1}  (\log\log D_K)^{3[K:\Q]/2+1} \\
& \ll_{[K:\Q],\ell}   |\mathrm{Cl}_K| \cdot V_K (\log D_K)^{-2\delta V_K} \log\log D_K
\end{align*}
from the inequality $\delta V_K \geq [K:\Q] \geq r_K-\rho_K$ and the definition of $V_K$ in \eqref{eqn:ClassNumberAssumption}. As 
\[
 V_K (\log D_K)^{-2\delta V_K}\log\log D_K \ll_{[K:\Q],\ell} (\log D_K)^{-\delta V_K}, 
 \]
this completes the proof of \cref{thm:Main}. 
\qed 
%

%%%%%%%%%
%%%%%%%%%
%%%%%%%%%
\section{Proof of Theorem 2}
\label{sec:SubconvexityProof}
%%%%%%%%%
%%%%%%%%%
%%%%%%%%%

We proceed similarly to the proof of \cref{thm:Main} but we shall take advantage of the assumption on integrality (instead of primality) appearing in \cref{lem:EV}. Let $K \neq \Q$ be a number field and $\ell \geq 2$ be an integer. Fix 
\begin{equation}
0 < \delta < \eta/2, \quad  x = D_K^{\frac{1-\delta/2}{2\ell([K:\Q]-1) }} 
\quad 
\label{eqn:FixParameters-Subconvexity}
\end{equation}
according to assumption \eqref{eqn:Assumption}. 
Define the integral ideal counting function
\begin{align*}
N_K^{\flat}(x) & := \sum_{n \leq x} \lambda^{\flat}_K(n) = |\{ \kn \subseteq \mathcal{O}_K : \N\kn \leq x, \N\kn \text{ squarefree}, \kp \mid \kn \implies  \kp \nmid \mathfrak{d}_K  \}| 
\end{align*}
so, by a full application of \cref{lem:EV},
\begin{equation} \label{eqn:Subconvexity-EVBound}
|\mathrm{Cl}_K[\ell]| \ll_{[K:\Q],\ell} \frac{\kappa_K}{N_K^{\flat}(x)}  D_K^{1/2}. 
\end{equation}

Given assumption \eqref{eqn:Assumption} with constant $A \geq 1$, we begin by defining the sum 
\[
S(x) := \sum_{n \leq x} \lambda^{\flat}_K(n) \varphi_k(n/x). 
\]
where $k = \lceil A \rceil + 1$ and $\varphi_k$ is given by \cref{lem:weight}. Since $\varphi_k(t) \leq 1$, it follows that
\begin{equation} \label{eqn:KeyStep2}
S(x) \leq \sum_{n \leq x} \lambda^{\flat}_K(n)  = N_K^{\flat}(x).	
\end{equation}
On the other hand, by  Mellin inversion, 
\begin{align*}
S(x) & = \frac{1}{2\pi i} \int_{(2)} \zeta_K(s) H_K(s,x) \frac{x^s}{(s+1)^{k+1}} ds \\
& = \frac{\kappa_K H_K(1,x)}{2^{[K:\mathbb{Q}]}} x  + \frac{1}{2\pi} \int_{-\infty}^{\infty}  \zeta_K(\tfrac{1}{2}+it) H_K(\tfrac{1}{2}+it,x)\frac{x^{1/2+it}}{(3/2+it)^{k+1}} dt . 
\end{align*}
By assumption \eqref{eqn:Assumption}, the upper bound on $H_K(\frac{1}{2}+it,x)$ from \cref{lem:Sift},  and the bound $\int_{-\infty}^{\infty}  (9/4+ t^2 )^{-1} dt \ll 1$, we deduce for some constant $C = C_{[K:\Q]} > 0$ that
\[
S(x) =  \frac{\kappa_K H_K(1,x)}{2^{[K:\mathbb{Q}]}} x  + O_{[K:\Q],\ell,A,\eta}\Big(  D_K^{\frac{1-\eta}{4\ell([K:\Q]-1)}}  x^{1/2} (e^{ (\log D_K)^{1/2}} \log x )^C  \Big). 
\]
Dividing both sides by $x$ and applying our choice in \eqref{eqn:FixParameters-Subconvexity}, we find that
\begin{equation}
\frac{S(x)}{x} = \frac{\kappa_K H_K(1,x)}{2^{[K:\Q]}} + O_{[K:\Q],\ell,A,\eta,\delta}\Big( D_K^{\frac{-\eta+\delta}{4 \ell ([K:\Q]-1)}} \Big). 
\label{eqn:Subconvexity-Penultimate}
\end{equation}

We now deduce \cref{thm:Subconvexity} by proving: 
\begin{enumerate}[label=(\roman*)]
	\item  an ineffective bound for  $\frac{1}{2}-\frac{1}{2\ell ([K:\Q]-1)} < \Delta < \frac{1}{2}$ and all number fields $K$;
	\item an effective bound for $\frac{1}{2}-\frac{1}{2\ell ([K:\Q]-1)} < \Delta < \frac{1}{2}$ and all $K$ without a quadratic subfield; 
	\item an effective bound for $\frac{1}{2}-\frac{\eta}{4\ell ([K:\Q]-1)} < \Delta < \frac{1}{2}$ and all number fields $K$. 
\end{enumerate}
For cases (i) and (ii),  we directly combine \eqref{eqn:Subconvexity-Penultimate} with \eqref{eqn:Subconvexity-EVBound} and \eqref{eqn:KeyStep2} to see that 
\[
|\mathrm{Cl}_K[\ell]| \ll_{[K:\Q],\ell} \frac{D_K^{1/2}}{x} \cdot \Big( H_K(1,x) + O_{[K:\Q],\ell,A,\eta,\delta}\Big( \kappa_K^{-1} D_K^{\frac{-\eta+\delta}{4\ell ([K:\Q]-1)}} \Big)  \Big)^{-1}.
\]
Since $x$ satisfies \eqref{eqn:FixParameters-Subconvexity} and  $H_K(1,x) \gg_{[K:\Q]} (\log\log D_K)^{-[K:\Q]/2}$ by \cref{lem:Sift}, it suffices to give a suitable lower bound for the residue $\kappa_K$. Case (i) follows by applying the ineffective Brauer--Siegel bound in the form 
\[
\kappa_K \gg_{[K:\Q],\ell,\delta} D_K^{\frac{-\delta/2}{2\ell ([K:\Q]-1)} }.
\]
Case (ii) follows by applying an effective lower bound of Stark \cite{Stark-EffectiveBrauerSiegel}   in the form 
\[
\kappa_K \gg_{[K:\Q]} (\log D_K)^{-1}.
\]
For case (iii), we instead give an upper bound for the residue. We  rearrange \eqref{eqn:Subconvexity-Penultimate} and again apply the lower bound on $H_K(1,x)$ from \cref{lem:Sift} to deduce  that
\[
\kappa_K \ll_{[K:\Q],\ell,A,\eta,\delta}  \frac{S(x)}{x} (\log\log D_K)^{[K:\mathbb{Q}]/2}   + D_K^{\frac{-\eta+\delta}{4 \ell ([K:\Q]-1)}} 
\]
for all number fields $K$. By \eqref{eqn:Subconvexity-EVBound}, we find that
	\[
		|\mathrm{Cl}_K[\ell]|
			\ll_{[K:\Q],\ell,A,\eta,\delta} \frac{D_K^{1/2}}{x} \frac{S(x)}{N_K^\flat(x)}(\log \log D_K)^{[K:\mathbb{Q}]/2} + D_K^{\frac{1}{2} -\frac{\eta-\delta}{4 \ell ([K:\Q]-1)}}.
	\]
 This yields case (iii) with our choice of $x$ in \eqref{eqn:FixParameters-Subconvexity} and that $S(x) \leq N^{\flat}_K(x)$ by \eqref{eqn:KeyStep2}. 
%~ 
\qed 

\begin{remark} The distinction between the outcomes in \cref{thm:Main,thm:Subconvexity} can be seen by comparing \eqref{eqn:KeyStep1} and \eqref{eqn:KeyStep2}. The subconvexity hypothesis \eqref{eqn:Assumption} in \cref{thm:Subconvexity} allows us to count integral ideals below the key threshold $x \approx D_K^{1/2\ell([K:\Q]-1)}$ from \cref{lem:EV}, so every non-zero term in $S(x)$ contributes to the desired savings via \eqref{eqn:KeyStep2}. Without this hypothesis, we can only count integral ideals below $x \approx D^{1/2}$ so the key threshold $y \approx D_K^{1/2\ell([K:\Q]-1)}$ in \cref{sec:MainProof} corresponds to the $y$-smooth norms $m$ (i.e., $p \mid m  \implies p \leq y$) contributing to $S(x)$ via \eqref{eqn:KeyStep1}. If very few $y$-smooth norms exist, then   the $y$-rough norms  $n$ (i.e., $p \mid n \implies p > y$) form the bulk of the contribution to $S(x)$ and do not necessarily provide any savings via the quantity $M$ in \cref{lem:EV}. This $y$-rough contribution appears to be our  bottleneck to bounding the residue $\kappa_K$ by more than $1/\log y$ in the unconditional setting of \cref{thm:Main}. 
\end{remark}

\begin{remark} \label{rem:Effective}
If $K$ contains a quadratic subfield, then one might ask whether the ineffective form of \cref{thm:Subconvexity} can still be made effective by again applying an effective Brauer--Siegel theorem due to Stark \cite{Stark-EffectiveBrauerSiegel}. We do not see how to do this immediately since \eqref{eqn:Subconvexity-Penultimate} requires 
\[
\kappa_K \gg_{[K:\Q],\ell,A,\eta,\delta} D_K^{\frac{-\eta+\delta}{4\ell([K:\Q]-1)} }
\]
 to ensure $S(x) \neq 0$ whereas Stark's theorem shows $\kappa_K \gg_{[K:\Q]} D_K^{-\frac{1}{[K:\Q]}}$ for such fields $K$.  As $\eta < 1$, this lower bound is insufficient for all integers $\ell \geq 2$. This lower bound barrier is notoriously difficult and tied to Landau--Siegel zeros. Our alternative argument allows us to avoid this requirement and hence establish an effective form of \cref{thm:Subconvexity}. 
 	
\end{remark}

%%%%%%%%%
%%%%%%%%%
%%%%%%%%%	

\bibliographystyle{alpha}
\bibliography{references}

\end{document}